\documentclass[12pt, final]{amsproc}
\usepackage{amssymb,amsbsy,amsmath,amsfonts,mathpazo}
\usepackage{latexsym,euscript,exscale}
\usepackage{mathrsfs}
\usepackage{url}

\usepackage{graphicx}
\usepackage{wrapfig}

\usepackage[all]{xy}
\usepackage[notcite,notref]{showkeys}

\usepackage{helvet}

  \newcounter{itemizedlistcounter}  

\author[F. \& J.\, Cabello S\'anchez]{F\'elix Cabello S\'anchez}
\author[]{Javier Cabello S\'anchez}
\address{Departamento de Matem\'aticas and IMUEx, Universidad de Extremadura\\
Avenida de Elvas, 06071-Badajoz, Espa\~na.
\newline
Orcid Id: 0000--0003--0924--5189, 0000--0003--2687--6193}
\email{fcabello@unex.es, coco@unex.es}

\newcommand{\eps}{\varepsilon}

\newcommand{\To}{\longrightarrow}

\newcommand {\R}{\mathbb R}

\renewcommand{\leq}{\ensuremath{\leqslant}}
\renewcommand{\geq}{\ensuremath{\geqslant}}

\newcommand{\Aut}{\operatorname{Aut}}

\newtheorem{thm}{Theorem}[section]

\newtheorem{theorem}[thm]{Theorem}

\newtheorem{lemma}[thm]{Lemma}
\newtheorem{prop}[thm]{Proposition}

\newtheorem*{claim*}{Claim}

\theoremstyle{definition}

\newtheorem{definition} [thm] {Definition}

\newtheorem{remark}[thm]{Remark}

\newtheorem*{prob*}{Problem}

\newtheorem{example}[thm]{Example}

\title[Dynamics of the semigroup of contractive automorphisms]{Dynamics of the semigroup of contractive automorphisms of Banach spaces}



\thanks{Supported in part by PID2019-103961GB-C21 and Junta de Extremadura, Project IB20038.}

\begin{document}

\begin{abstract}
Motivated by some recent twaddles on Mazur rotations problem, we study the 
``dynamics'' of the semigroup of contractive automorphisms of Banach spaces, 
mostly in finite-dimensional spaces. We focus on the metric aspects of the ``action'' 
of such semigroups, the size of the orbits and semitransitivity properties, 
and their impact on the geometry of the unit ball of the underlying space.
\end{abstract}

\noindent{{\tiny{\boxed{\text{Version \today}}}}

\bigskip

\maketitle

\section{Introduction}

This paper studies the ``dynamics'' of the semigroup of contractive automorphisms of finite dimensional normed spaces.
Our interest in this subject, and even the topic itself, stems from Mazur rotations problem: \emph{Is every separable Banach space whose group of linear isometries acts transitively on the unit sphere isometric (or isomorphic) to a Hilbert space?}  (Cf. \cite[remarque \`a la section 5 du chapitre XI]{Banach}.)

We are not going to go into this issue, firstly because there is a very recent survey paper on the subject and secondly because we believe that our study 
is (moderately) interesting in its own right.  In any case, we recommend the reader to have a look at the papers \cite{becerra, CFR, BRP,  CDKKLM}.

We consider only real spaces; most of the time we work in  finite dimensions, often in the plane.
In particular $\ell_p^n$ denotes $\R^n$ with the norm $\|x\|_p= \big(\sum_{i=1}^n |x_i|^p \big)^{1/p}$ for $1\leq p<\infty$, while $\|x\|_\infty=\max_{i\leq n}|x_i|$, and $e_i$ denotes the corresponding unit vector of $\R^n$.

Given  a (real) normed space $X$, its (closed unit) ball is the set $B=\{x\in X:\|x\|\leq 1\}$ and $S=\{x\in X:\|x\|= 1\}$ is the unit sphere. We write $L(X)$ for the space of all (bounded, linear) endomorphisms of $X$ with the operator norm
$\|T\|=\sup_{\|x\|\leq 1}\|T(x)\|$.
 An operator $T$ on $X$ is an \emph{automorphism} if there is  $L\in L(X)$ such that $TL=LT={\bf I}_X$. An operator $T$ is \emph{contractive}, or a contraction, if $\|T\|\leq 1$. This means that $TB\subset B$.

Clearly, the contractive automorphisms of $X$ form a semigroup of $L(X)$ that we will denote by $\Aut_1(X)$. We are interested in the ``action'' of $\Aut_1(X)$ on the unit sphere of $X$, especially in the size of the ``orbits''
 $
O(x)=\{y\in S: y=Tx \text{ for some }T\in \Aut_1(X)\}
$ 
for $x\in S$. Note that if $\|y\|<\|x\|$ then there is always $T\in \Aut_1(X)$ such that $y=Tx$, so considering only points of $S$ is fine. Life without local convexity can be much harder, see \cite{roberts} for examples of {\em rigid} quasi Banach spaces.

The space $\ell_1^2$ already provides a quite interesting example (see Figure~1 and the comments following Theorem~\ref{th:ST}):

\begin{itemize}
\item $O(x)=S\iff x=\pm e_i$ for some $i=1,2$.
\item $O(x)=S\backslash\{\pm e_1, \pm e_2\}$ for any other $x\in S$.
\item For each $x\in S$ the set $O(x)$ is a neighbourhood of $x$ relative to $S$.
\end{itemize}

\section{Semitransitivity}

A Banach space $X$ is said to be semitransitive (ST) if for every $x,y\in S$ 
there is $T\in \Aut_1(X)$  such that $y=Tx$. Hilbert spaces are ST: 
actually they are even {\em transitive}, that is, $T$ can be taken to be 
isometric, that is, $\|T\|=\|T^{-1}\|=1$. This quickly follows from the 2D case and 
the existence of orthogonal complements. Any other example in sight? 
Not yet. To understand what the issue is really about 
we need to introduce a couple of definitions.  

\begin{wrapfigure}{r}{0.45\textwidth}
  \begin{center}\vspace{-25pt}
    \includegraphics[width=0.45\textwidth]{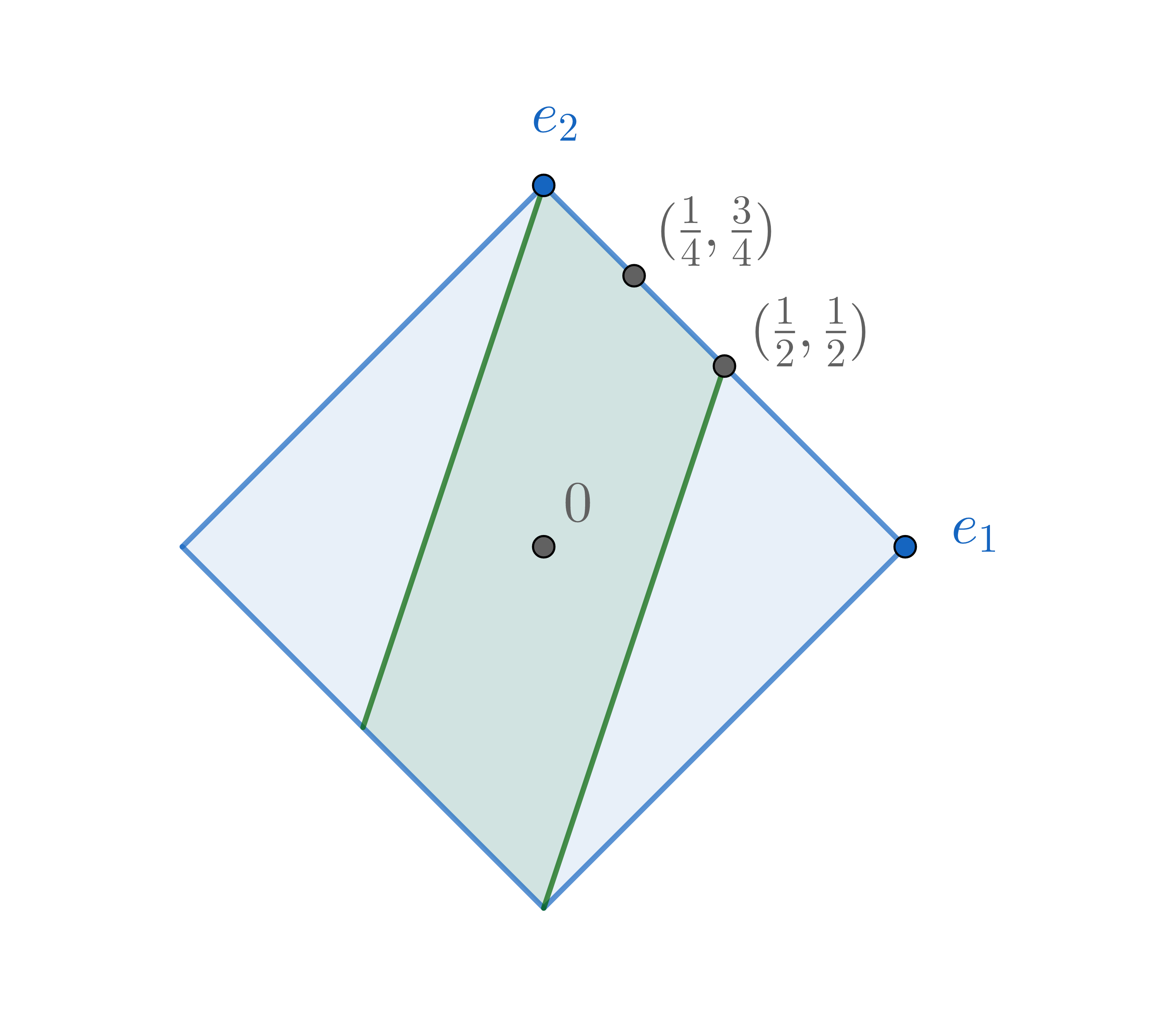}
  \end{center}
\caption{The matrix $\frac{1}{2} \binom{1 \; 0}{1 \; 2}$ implements a contractive automorphism of $\ell_1^2$ that sends $(\frac{1}{2}, \frac{1}{2})$ 
to  $(\frac{1}{4}, \frac{3}{4})$. If we fix $e_1$ and let $e_2$ slip towards 
$e_1$ we can send $(\frac{1}{4}, \frac{3}{4})$ back to $(\frac{1}{2}, \frac{1}{2})$, using $\frac{1}{3} \binom{3 \; 1}{0 \; 2}$.}
\end{wrapfigure}

Let $X$ be a finite dimensional (FD) space (or even a Banach space isomorphic 
to a Hilbert space) with unit ball $B$ and unit sphere $S$. An ellipsoid in 
$X$ is the closed unit ball (centered at the origin) of an Euclidean norm 
equivalent to the original norm of $X$. Ellipsoids are invariably assumed 
to be centered at the origin. If $X=\R^n$ with some norm, then a disc 
(centered at $x$) is a set of the form $D=\{y\in\R^n: \|y-x\|_2\leq r\}$, 
where $r>0$ and $\|\cdot\|_2$ is the {\em usual} Euclidean norm.

We say that an ellipsoid $E$ is inner at $x\in S$ if $E\subset B$ and $x\in E$. 
We say that $E$ is outer at $x$ if $B\subset E$ and $x\in\partial E$. The 
following duality argument will be used over and over without further mention: 
if $x\in X$ and $x^*\in X^*$ are such that $\|x\|=\|x^*\|= \langle x^*,x\rangle =1$, 
then $E$ is inner (respectively, outer) at $x$  if and only if the dual ellipsoid 
$E^*=\{y^*\in X^*: |\langle y^*, y\rangle |\leq 1\; \forall\, y\in E\}$ is outer 
(respectively, inner) at $x^*$.

Those (nonzero) points where the norm is (G\^ateaux) differentiable are called smooth.
This happens to $x$ if and only if there is exactly one functional $x^*$ such 
that $\|x\|=\|x^*\|$ and $ \langle x^*,x\rangle =\|x\|\|x^*\|$. Such an $x^*$ is called a support 
functional,  necessarily agrees with  $\tfrac{1}{2}d_x\|\cdot\|^2$, and is often 
denoted by $J(x)$. If the norm is differentiable off the origin the space $X$ itself 
is called smooth and the mapping $J:X\To X^*$ sending each point into its unique 
support functional (and $0$ to $0$) is called the duality map.

It is clear that points that admit inner ellipsoids are smooth, while those 
that have outer ellipsoids are extreme. 
The following remark is obvious, yet very useful:

\begin{lemma}\label{lem:yinOx}
Let $X$ be a finite dimensional space and assume $y\in O(x)$.
\begin{itemize}
\item If $x$ is a smooth point of $S$ then so is $y$.
\item If $y$ is an extreme point of $B$ then so is $x$.
\item If $x$ admits an inner ellipsoid, then so $y$ does.
\item If $y$ admits an outer ellipsoid, then so $x$ does.
\end{itemize}
\end{lemma}

The following result is almost obvious, just using compactness:

\begin{lemma}\label{lem:hay}
In a finite dimensional space, the set of points admitting inner ellipsoids is 
dense in the unit sphere and the set of points admitting outer ellipsoids is nonempty.
\end{lemma}

See \cite[Lemma 3.4]{wheeling} for an extension to infinite dimensional spaces. 
The following characterization is in turn a particularization of \cite[Proposition 3.5]{wheeling}.

\begin{theorem}\label{th:ST}
For a finite dimensional space $X$ the following are equivalent:
\begin{itemize}
\item[{\rm (a)}] $X$ is semitransitive. 
\item[{\rm (b)}] Every point of the unit sphere admits inner and outer ellipsoids.
\item[{\rm (c)}] $X^*$ is semitransitive. 
\end{itemize}
\end{theorem}

\begin{proof}
The implication (a)$\implies$(b) is clear after Lemma~\ref{lem:yinOx} and \ref{lem:hay}.
The converse (b)$\implies$(a) follows from the fact that {\em if $y\in S$ admits 
an inner ellipsoid and $x\in S$ admits an outer ellipsoid, then $y\in O(x)$}. 
Indeed, let $E$ be inner at $y$ and $F$ outer at $x$. Since all Hilbert spaces 
of the same dimension are isometric there is $L\in L(X)$ which maps $F$ onto $E$. 
Clearly $Lx\in \partial E$  and by transitivity there is an isometry $R$ of the 
norm associated to $E$ such that $R(Lx)=y$. The composition $T=RL$ is a contraction 
since $TB\subset B$.

On the other hand, by the remark preceding Lemma~\ref{lem:yinOx},  $X$ satisfies (b) if and only if $X^*$ does, hence (b) is equivalent to (c) too.
\end{proof}

The above criterion shows that for $n\geq 2$ and $p\neq 2$ the spaces $\ell_p^n$ are not ST (if $p<2$ the unit vectors do not have inner ellipsoids;  if $p>2$ they do not have outer ellipsoids). The paper \cite{wheeling} contains examples of ST norms on the plane which were constructed with slightly different purposes, using sledgehammers to crack nuts. Example~\ref{ex:splicing} below is much simpler; Example~\ref{ex:noBT}  exhibits a rather unexpected behaviour.

Let us call $y\in S$ {\em flat} if there exists some homogeneous hyperplane $H\subset X$  
such that $(y+H)\cap S$ is a neighbourhood of $y$ relative to $S$ ---note that $H=\ker J(y)$. If $y$ is flat, then $y\in O(x)$ for all $x\in S$. Indeed, let $x^*$ be a norm-one functional such that $\langle x^*, x\rangle=1$, put $H_x=\ker x^*$, and let $L:H_x\To H$ be any linear isomorphism. If $\eps>0$ is sufficiently small, the automorphism $T$ given by $T(cx+u)=cy+\eps L(u)$, where $c\in\R, u\in H_x$, is   contractive ---and sends $x$ to $y$, of course.

Call $x\in S$ pilgrim if $O(x)$ is dense in $S$. It is clear from Lemma~\ref{lem:hay} that 
every FD normed space has pilgrim points. As for the ``size'' of the set of 
pilgrim points it is nearly obvious that $X$ is strictly convex and smooth, then
the set of pilgrim points is dense in the unit sphere of $X$. See why? 
 On the other hand, we have just seen that the orbit 
of any flat point consists precisely of the set of flat points of $S$ and this means 
that in any polyhedral space the set of pilgrim points is dense in the sphere. The same is true for spaces whose unit ball is the the intersection of two ellipsoids.

\begin{wrapfigure}{r}{0.45\textwidth}
  \begin{center}
   \vspace{-20pt} \includegraphics[width=0.5\textwidth]{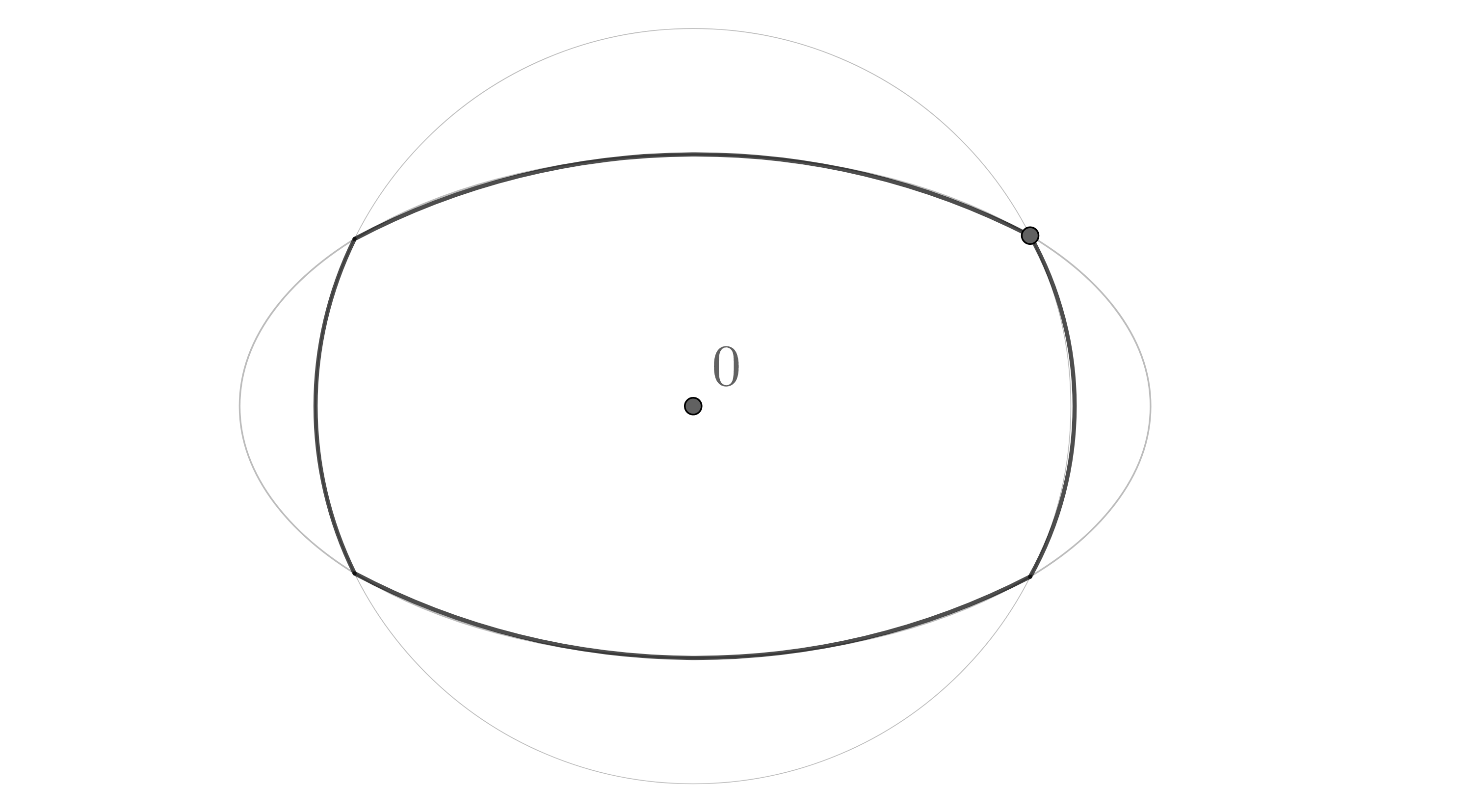}
  \end{center}
\caption{Intersection of two ellipsoids}
\end{wrapfigure}

It is also clear that a FD space whose set of pilgrim points is dense in the unit 
sphere must must have some kind of regularity: if we mix 
a polyhedral norm with some strictly convex norm, we loose the density of 
the set of pilgrim points ---think of the norm in $\R^2$ that agrees with $\|\cdot\|_2$ 
in the first quadrant and with $\|\cdot\|_1$ in the second one.

The following example shows that some natural conjectures that one might come up with are false:

\begin{example}
{A uniformly convex and smooth normed space whose sphere is not the orbit of any point}.
Pick $1<p<2<q<\infty$ and define a norm on $\R^2$ letting
\begin{equation*}
\|x\|=\begin{cases}
\|x\|_p &\text{ if } x_1x_2\geq 0,\\
\|x\|_q &\text{ if } x_1x_2\leq 0.
\end{cases}
\end{equation*}
All points on the sphere admit both inner and outer ellipsoids, except $\pm e_i$ which  have neither inner ellipsoids ($p$ is too small) nor outer ($q$ is too large). The inexorable conclusion is that if $x\neq \pm e_i$, then neither $x$ can be in the orbit of $\pm e_i$ nor $\pm e_i$ in the orbit of $x$.
\end{example}

\begin{wrapfigure}{r}{0.45\textwidth}
  \begin{center}
  \vspace{-10pt}
    \includegraphics[width=0.35\textwidth]{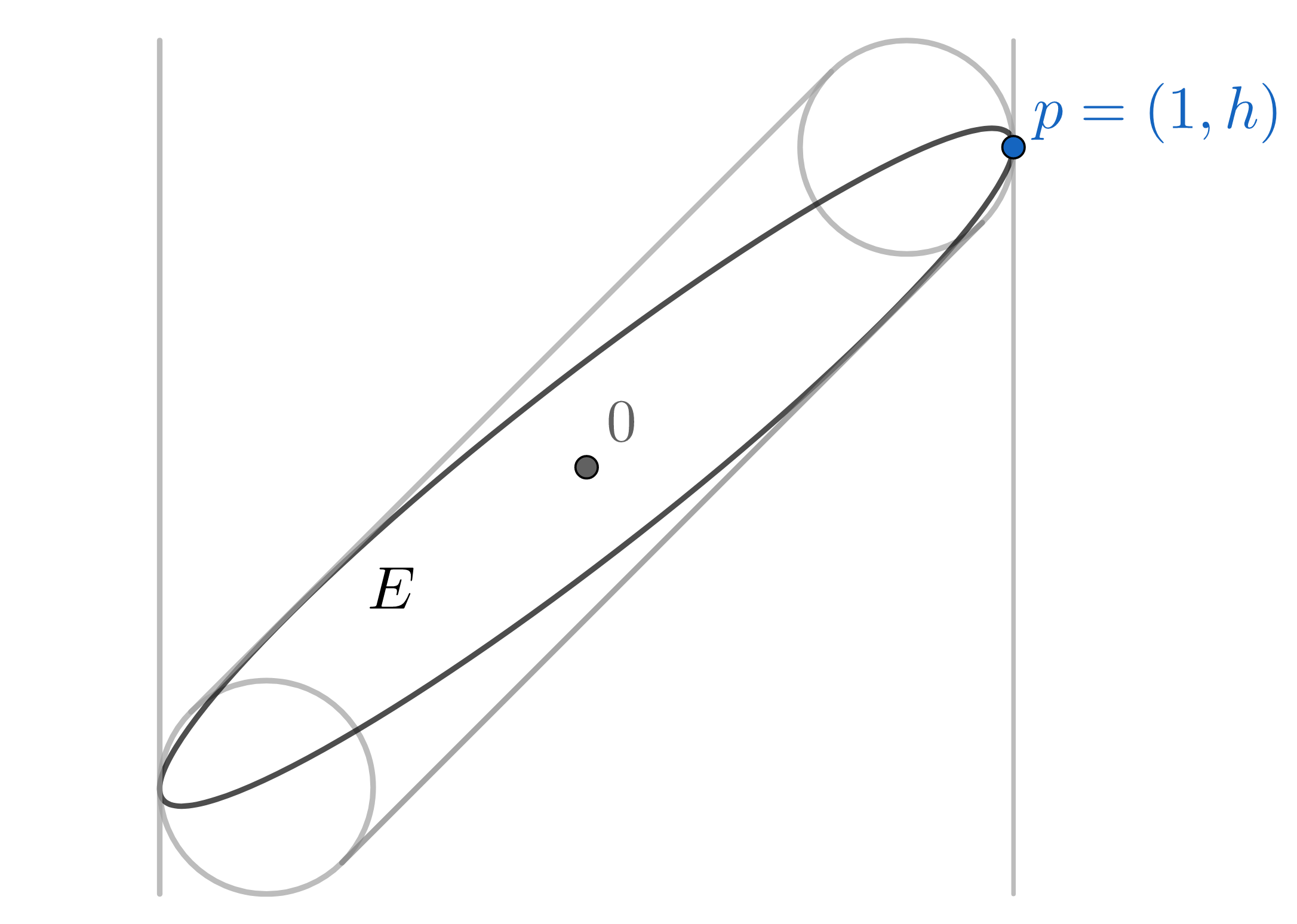}
  \end{center}
  \caption{The construction of Lemma~\ref{lem:oE}}
\end{wrapfigure}

The notion of curvature (of a plane curve) is very useful both for understanding 
the geometry of the norm and for checking the ST character of norms in the plane. 
As far as this paper is concerned, everything we need to know about curves can 
be found in the lovely book by the Bucks \cite[\S 8.4]{buck}. Let $I$ be an 
interval and $\Phi:I\subset\R\To \R^n$ be a rectifiable curve parameterized by arc length. 
If $x=\Phi(t)$, and $\Phi$ is twice differentiable at $t$, then the curvature 
of $\Phi$ at $x$ is defined as 
 $ 
 \varkappa(\Phi, x)=\|\Phi''(t)\|_2.
 $ 
We are aware that there is a small problem here if $x$ is a double (or multiple) 
point of $\Phi$ since  $ \varkappa(\Phi, x)$ depends not only on $x$, but also on $t$. 
We will not insist on this point. Note also that this notion of curvature refers 
to the usual Euclidean norm, although other approaches are possible 
(see \cite[\S 4.3]{tingley}). 

If $\Phi$ is the graph of a function $f:I\To \R$, in the sense that $\Phi(t)=(t,f(t))$ 
for $t\in I$, then
\begin{equation}\label{eq:kvsf''}
\varkappa(\Phi,x)=\frac{|f''(t)|}{\big(1+f'(t)^2\big)^{3/2}},
\end{equation}
where $x=(t,f(t))$, and we no longer assume $\Phi$ parameterized by its arc length
(see Bucks' \cite[Exercise 5 on p. 416]{buck}). As a consequence:

\begin{lemma}\label{lem:dA<dB}
Let $A, B$ be convex domains in $\R^2$ with (piecewise) smooth boundaries 
$\partial A, \partial B$. If $\partial A$ and $\partial B$ are tangent at $p$ (with 
the same inward-pointing normal) and $\kappa(\partial A, p)< \kappa(\partial B, p)$ 
then there is a neighbourhood  $U$ of $p$ such that $A\cap U\subset B\cap U$.
\end{lemma}

\begin{proof}
In view of (\ref{eq:kvsf''}) this is a restatement of the fact that if $f$ is 
a real-valued function of a single variable with $f(t)=f'(t)=0$ and $f''(t)>0$, 
then $f$ is non-negative in some neighbourhood of $t$.
\end{proof}

Since a circle of radius $r$ has curvature $r^{-1}$ at all points the Lemma 
implies that if $S$ has finite curvature at $x$, then $x$ belongs to a disc $D$ 
which is contained in $B$. What is more interesting than it may seem:

\begin{lemma}\label{lem:oE}
Let $B$ be the unit ball of a norm on $\mathbb R^2$ and  let $x\in\partial B$. Then 
$x$ admits an inner ellipsoid (ellipse) if and only if there is a disc $D$ such that $x\in D$ and $D\subset B$. 

Dually, $x$ admits an outer ellipsoid if and only if there is a disc $D$ such that $x\in \partial D$ and $B\subset D$. 
\end{lemma}

\begin{proof}[Sketch of the proof]
We only prove the first part, being the ``only if'' part clear: if $E$ is an ellipse of semi-axes $a\leq b$
and  $x\in \partial E$, then $E$ contains the disc of radius $a^2/b$ centered at the inner normal and passing through $x$.

To prove the converse it suffices to check the following: {\em Let $h\geq 0$ and $0<r<1$. 
Let $D=D((1-r, h); r)$ and $p=(1,h)$. Then there is an ellipsoid that passes through $p$ and is contained in the convex hull of $-D\cap D$}.

Indeed, the equation of any ellipse with center the origin is
\begin{equation*}
f(x,y) \stackrel{{\rm def}}{=\hspace{-2pt}=} Ax^2+By^2+Cxy= 1.
\end{equation*}
If besides the tangent line at $p=(1,h)$ is ``vertical'' we have
$$
\begin{cases}A+Bh^2+Ch=1\\
2Bh+C=0
\end{cases}\quad\implies \quad \begin{cases}B= -C/(2h)\\
A=1-\frac{1}{2}Ch
\end{cases}
$$
since $\nabla f(x,y)= (2Ax+Cy, 2By+Cx)$.

On the other hand (see Buck's \cite[Exercise 6 on p. 416]{buck}), the curvature of a plane curve satisfying the equation $f(x,y)=1$ is given by the formula
$$
\varkappa = \frac{\left| f_{xx}f_y^2 -2f_{xy}f_xf_y +  f_{x}^2f_{yy} \right|}{\left( f_{x}^2+ f_{y}^2	\right)^{3/2}} .
$$
Let us compute the curvature at $p$. Clearly, we have $f_y(p)=0, f_{yy}=-C/(2h), f_x(p)= 2$, so 
$$
\varkappa(p) = \frac{\left|f_{x}^2f_{yy}\right|}{\left|f_{x}^3\right|}= \frac{|C|}{4h}.
$$
Hence taking $|C|$ large enough does the trick.
\end{proof}

\begin{wrapfigure}{r}{0.48\textwidth}
  \begin{center}
    \includegraphics[width=0.48\textwidth]{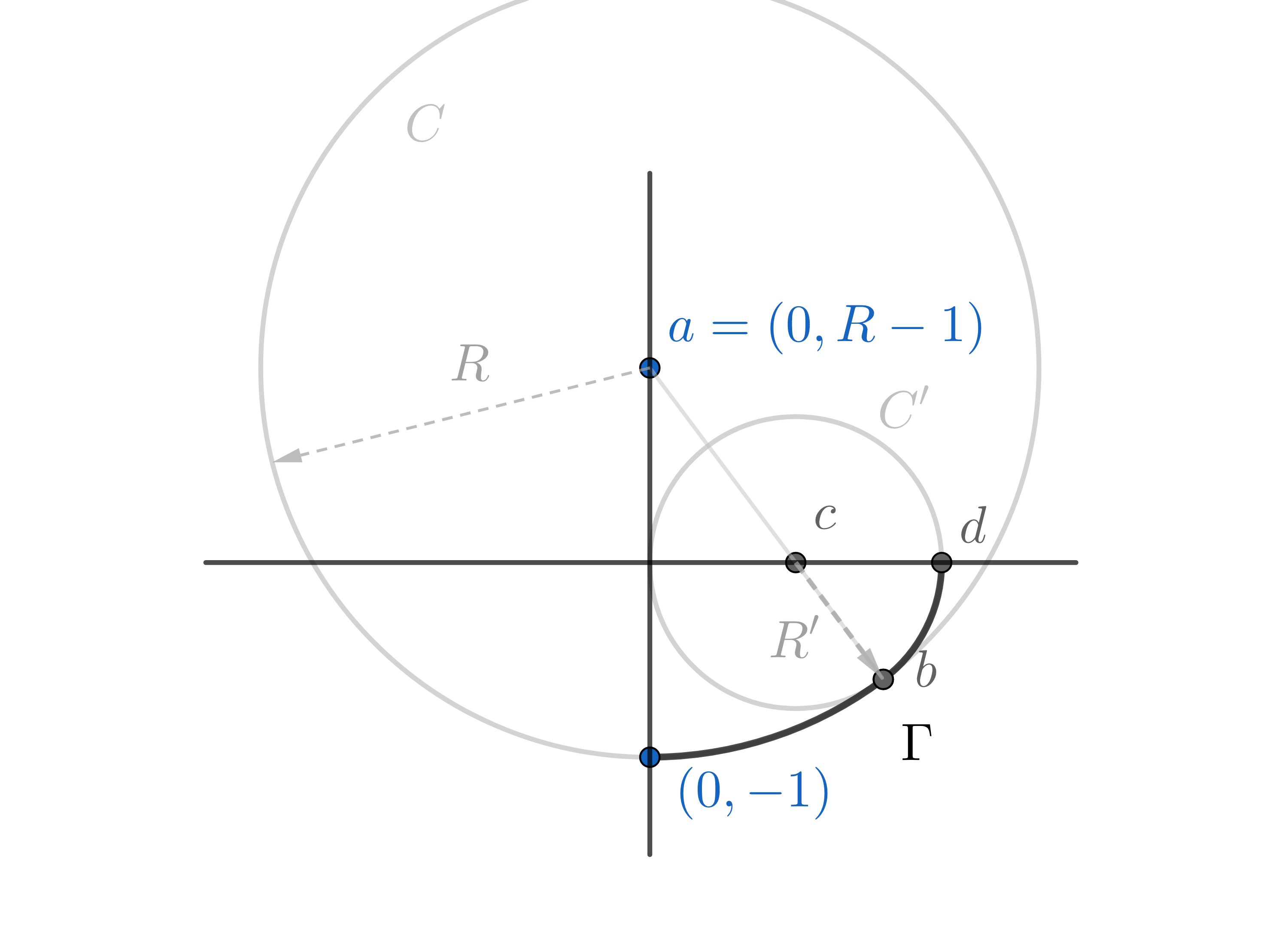}
  \end{center}
  \caption{The construction of Example~\ref{ex:splicing}, which is also the basis of Example~\ref{ex:noBT}}
\end{wrapfigure}

\begin{example}\label{ex:splicing}
Semitransitive spaces splicing arcs of circles.
\end{example}
Here is a simple example of ST norm on $\R^2$. Fix any $R>1$ and let $C$ be the 
circle of center $a=(0,R-1)$ and radius $R$. Let $b=(b_1,b_2)$ be any point of 
$C$ such that $b_1>0$ and $b_2<0$. Let $c$ be the intersection of the horizontal 
axis and the line joining $a$ and $b$. Let $C'$ be the circle of center $c$ and 
radius $R'=\|c-b\|_2$. Note that $C'$ passes through $b$ and that $R'<R$. Let $d$ 
be the rightmost point in the insersection of $C'$ and the horizontal axis. 
Let $\Gamma$ be the curve (in the fourth quadrant) that agrees with $C$ from 
$(0,-1)$ to $b$ and with $C'$ from $b$ to $d$.

It is clear that there is exactly one norm $\|\cdot\|$ on the plane such that 
$\|(x,y)\|=1 \iff (|x|,-|y|)\in\Gamma$ and also that $\|\cdot\|$ is ST since 
every point in the unit sphere satisfies the criteria provided by Lemma~\ref{lem:oE}.

\medskip

Let $X$ be a FD space and assume that the orbit of $x$ is $S$. Put
$$
O(x,K)=\big\{y\in S: y=Tx \text{ for some } T\in \Aut_1(X) \text{ such that }\|T^{-1}\|\leq K\big\}
$$
Then $S=\bigcup_{n=1}^\infty O(x,n)$. An obvious ``category'' argument and the fact that the sets $\{T\in \Aut_1(X): \|T^{-1}\|\leq K\}$ are compact in $L(X)$ show that for every open set $U\subset S$ there is $K$ such that $U\cap O(x,K)$ has nonempty interior in $S$.

This fosters the idea that ST spaces could be {\em boundedly semitransitive} (BST) in the following sense: there is a constant $K$ such that, for every $x,y\in S$ there is $T\in\Aut_1(X)$ such that $y=Tx$ and $\|T^{-1}\|\leq K$.

Unfortunately this is not the case in general:

\begin{example}\label{ex:noBT}
A (2-dimensional) semitransitive space which is not boundedly semitransitive.
\end{example}

Recall that given a locally integrable function $k:I\To [0,\infty)$, where $I$ is an interval containing the origin, the formul\ae
\begin{align*}
x(s)&=x_0+\int_0^s \cos\left(\int_0^t k(u)\,du\right) dt,\\
y(s)&=y_0+\int_0^s \sin\left(\int_0^t k(u)\,du\right) dt
\end{align*}
define a $C^1$ plane curve $\Gamma(s)=(x(s), y(s))$ such that
\begin{itemize}
\item $\|\Gamma'(s)\|_2=1$ for all $s\in I$.
\item If $s$ is a Lebesgue point of $k$, then the curvature of $\Gamma$ at $s$ is $k(s)$.
See \cite[Definition 7.1.19]{kannan}
\end{itemize}

\begin{wrapfigure}{r}{0.45\textwidth}\label{fig:k}
  \begin{center}
  \vspace{-0pt}
    \includegraphics[width=0.45\textwidth]{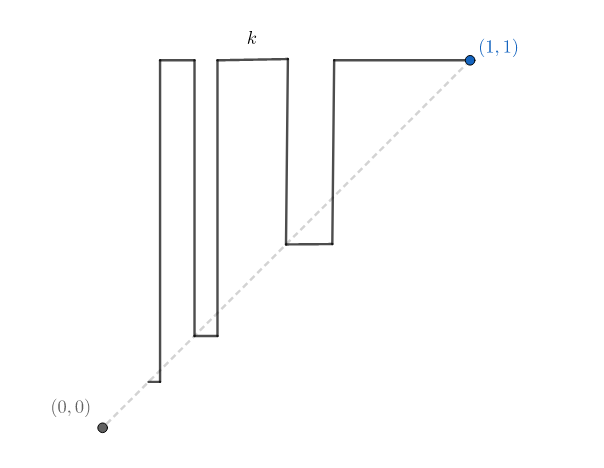}
  \end{center}
  \caption{A portion of the ``graph'' of the curvature function $k$ of Example~\ref{ex:noBT}}
\end{wrapfigure}

Let us first define $k:[-\frac{\pi}{2}, 1]\To(0,1]$ as follows (see Figure~5):
\begin{itemize}
\item For each $n\in\mathbb N$ the value of $k$ on $[2^{-n},2^{-n}+ 2^{-n-2}]$ is $2^{-n}$.
\item Otherwise $k(s)=1$.
\end{itemize}
Then we consider the curve defined as above, with $x_0=-1, y_0=0$, and call it $\Gamma$.
Now look at the endpoint $p=(x(1), y(1))$ of $\Gamma$ corresponding to $s= 1$: a moment's reflection suffices to realize that:

\begin{itemize}
\item The initial point of $\Gamma$ is $(-1,0)$.
\item $0<x(1)< 1, -1<y(1)<0$ since $(x(0), y(0))= (0,-1)$, both $x'(s)$ and 
$y'(s)$ are strictly positive for $0<s<1$ and $s$ parametrizes $\Gamma$ naturally.
\item The (semi) tangent line to $\Gamma$ at $p$  has strictly positive slope and, 
moreover, if $(a,0)$ is the only point of that line that lies on the ``horizontal'' axis,  then $a>1$.
\end{itemize}

It follows (see Figure~4) that there exist a pair of circles $C, C'$ such that $C$ is tangent 
to $L$ at $p$, $C'$ is tangent to the line $x=1$ at $(1,0)$ and $C$ and $C'$ 
are mutually tangent at some other intermediate point, say $q$.

Let $S^-$ be the curve that ``agrees'' with $\Gamma$ between $(-1,0)$ and $p$, 
with $C$ between $p$ and $q$ and with $C'$ between $q$ and $(1,0)$.  It is clear 
that there is exactly a norm, say $\|\cdot\|$, on $\R^2$ 
whose unit sphere $S$ contains $S^-$.

We want to see that the resulting space is ST. Observe that $S^-$ is a (``smooth'') 
countable union of arcs of circles: going from right to left we find first the arcs 
$C'$ and $C$, then an arc of circle of radius 1 and lenght 
$\frac{1}{2}- \frac{1}{2^3}= \frac{3}{2^3}$, then an arc of circle of radius $2$ 
and lenght $\frac{1}{8}$, then another  arc of circle of radius 1 and lenght 
$\frac{3}{2^4}$, then an arc of circle of radius $4$ and lenght $\frac{1}{16}$ 
and so on. These arcs accumulate at the point $(0,-1)$ in the obvious sense. 
Finally we have an arc of radius 1, centered at the origin, between $(-1,0)$ and $(0,-1)$.

Now it should be clear that all points of $S^-$ admit inner discs and also that 
all of them, except $(0,-1)$, have outer discs. But actually $(0,-1)$ admits an 
outer disc whose  radius is not very large. Indeed for every $n$ the (Lebesgue) 
measure of the set $\{t\in[0,2^{-n}]: k(t)=1\}$ is at least $\frac{3}{4}2^{-n}$. 
Therefore, if we put $K(s)=\int_0^s k(t)dt$, then $\frac{3}{5}s\leq K(s)\leq s$ 
for $0\leq s\leq 1$. Since
$$
x(s)=\int_0^s \cos K(t)\, dt\quad\text{and}\quad
y(s)=-1+\int_0^s \sin K(t)\, dt
$$
it is clear that $\Gamma(s)$ cannot leave the disc of radius $\frac{5}{3}$ 
centered at $(0,-1)$ for any $0\leq s\leq 1$, which is enough.

The resulting normed plane, whose sphere is $S$, cannot be BST because the 
inverse of a contractive automorphism sending a point with small curvature 
to a point of large curvature must have large norm, see why? If you still 
do not see it clearly, take a look at the upcoming Proposition~\ref{prop:BST}. 


\section{Bounded semitransitivity}

We now study BST spaces. It turns out that these admit a quite elegant geometric 
characterization. Before going any further, recall that the modulus of uniform 
convexity of a Banach space $X$ is the function $\delta_X:(0,2]\To[0,1]$ defined by
\begin{equation}\label{eq:defdX}
\delta_X(\eps)=\inf\left\{1-\left\|\frac{x+y}{2}\right\|: \|x-y\|\geq \eps, \|x\|,\|y\|\leq 1\right\}.
\end{equation}
Note that $1-\frac{1}{2}\left\|x+y\right\|$ represents the distance between the 
midpoint of $x$ and $y$ and its closest multiple in the unit sphere (the larger 
the MUC the more convex the ball is) and that one can replace all inequalities 
by equalities in (\ref{eq:defdX}) without altering the value of  $\delta_X(\eps)$. 
We say that $\delta_X$ is of power type 2 if $\delta_X(\eps)\geq c\eps^2$ for 
some $c>0$ and all $0<\eps\leq 2$.

Hilbert spaces have ``optimal'' MUC given by
$$\delta_2(\eps)=1-\sqrt{1-\left(\frac{\eps}{2}\right)^2}\geq \frac{ \eps^2}{8}.$$
Condition (d) below is stated also in terms of the modulus of uniform smoothness 
(MUS, a measure of the  ``flatness'' of the unit ball which we do not even define here). 
It will suffice to recall here that $X$ has MUS of power type 2 if and only if $X^*$ 
has MUC of power type 2, so that (d) is equivalent to: \emph{Both $X$ and $X^*$ have 
MUC of power type $2$}. We refer the reader to \cite[\S~1.e]{LT2} for the basics on uniform convexity and smoothness.

\begin{prop}\label{prop:BST}
For a finite dimensional space $X$ the following are equivalent:
\begin{itemize}
\item[{\rm (a)}] $X$ is boundedly semitransitive. 
\item[{\rm (b)}] $X^*$ is boundedly semitransitive. 
\item[{\rm (c)}] There is a constant $\lambda$ such that every $x\in S$ admits an 
inner ellipsoid $E$ and an outer ellipsoid $F$ such that $F\subset \lambda E$.
\item[{\rm (d)}] $X$ has moduli of uniform convexity and smoothness of power type $2$.
\end{itemize}
\end{prop}

The remainder of the section is devoted to proving this result. Let's get the 
boring parts out of the way now. The equivalence (a)$\iff$(b) can be proved as 
in Theorem~\ref{th:UMST} since taking Banach space adjoints preserves 
automorphisms, inverses, and the operator norm.

(a)$\implies$(c). Assume $X$ is BST with constant $\beta$. Fix $y\in S$ and let 
$E$ and $F$ be inner and outer ellipsoids at $y$, respectively, which exists by 
ST and Theorem~\ref{th:UMST}. Clearly  $F\subset \mu E$ for sufficiently large 
$\mu$. Pick $x\in S$ and then $T,L\in\Aut_1(X)$ such that $x=Ty, y=Lx$, with 
$\|T^{-1}\|, \|L^{-1}\|\leq \beta$. Then $E'=TE$ is inner at $x$, $F'=L^{-1}F$ 
is outer at $x$, and $F'\subset \beta^2\mu E'$.

The implication (c)$\implies$(a) can be proved as the implication (b)$\implies$(a) 
of Theorem~\ref{th:UMST}, just following the track of the norms.

We now adress  (c)$\implies$(d). The proof uses the so-called {\em modulus of 
strong extremality}, defined for $x\in S$ and $\varepsilon\in(0,1]$ as:
$$
\Delta_X(x,\eps)=\inf\{1-\rho: \text{there is $y\in X$ such that }\|y\|\geq \eps \text{ and } \|\rho x\pm y\|\leq 1\}.
$$
It is not hard to see that 
$
\delta_X(2\eps)=\inf_{\|x\|=1} \Delta_X(x,\eps)
$ ---only the easiest part ``$\geq$'' will be used here.

Assume $x\in S$ admits an outer ellipsoid $F$ with $F\subset\alpha B$. This 
implies that  if $|\cdot|$ denotes the Euclidean norm corresponding to $F$, 
then $|\cdot|\leq \|\cdot\|\leq \alpha|\cdot|$. An easy computation based 
on the fact that $S$ and $\partial F$ are tangent at $x$ shows that 
$$
\Delta_X(x,\eps)\geq \Delta_2(x,\eps/\alpha)= 1-\sqrt{1-\left(\frac{\eps}{\alpha}\right)^2}\geq \frac{\eps^2}{2\alpha^2},
$$
where, as the reader may guess, $\Delta_2$ stands for the modulus of 
strong extremality of the 
Euclidean norms. Thus, (c) implies that $\delta_X(\eps)\geq c\eps^2$ 
where $c$ depends only on $\lambda$. The fact that $X$ has MUS of power type 2 
follows from the fact that a finite dimensional space satisfies (c) if and only 
if its dual does since a Banach space has MUS of power type 2 if and only if 
its dual has MUC of power type 2.

It only remains to check that (d)$\implies$(c), which amounts to see that if 
$X$ has MUC of power type 2, then there is a constant $\alpha>0$ such that 
every $x\in S$ admits an outer ellipsoid $F$ such that $F\subset \alpha B$.

The proof is based on the following, surely well-known, observation:


\begin{lemma}
Let $X$ be a normed space with modulus of uniform convexity $\delta$. Let $x\in S$ 
and assume that $H$ is a supporting hyperplane at $x$. Assume $z\in B$ and write 
$z=tx+u$, with $t\in\R$ and $u\in H$. Then $1-t\geq \delta(\|u\|)$ and so $t^2+ \delta(\|u\|)\leq 1$.
\end{lemma}

\begin{proof}
One can assume that $X$ has dimension 2 and also that $\|z\|= 1$. Since 
$\|tx+u\|\geq \|tx\|=|t|$ we have $|t|\leq 1$ and so $\|u\|\leq 2$. In particular 
$\delta(\|u\|)$ is correctly defined and since $\delta(\eps)\leq 1$ for all 
$\eps\leq 2$ we may assume $0<t<1$.

We encourage the reader to draw their own monikers. Let $u^-$ be the only 
negative multiple of $u$ such that $\|tx+u^-\|=1$. We claim that even 
$1-t\geq \delta(\|u\|+\|u^-\|)$, which is clear just applying the definition 
of MUC to $\eps= \|u\|+\|u^-\|$ which equals the distance between $x'=tx+u$ 
and $y'=tx+u^-$, taking into account that the midpoint of $x'$ and $y'$ lies 
in the line $\{tx+su: s\in\R\}$ and that  $\|tx+su\|\geq |t|$ for all $s$. 
In particular $1-\frac{1}{2}\|x'+y'\|\leq 1-t$. Nice, isn't it?
\end{proof}

To complete the proof,
let $|\cdot|$ be an Euclidean norm on $X$ such that $|\cdot|\leq \|\cdot\|\leq C|\cdot|$ 
for some constant $C$, for instance one could take the norm associated to the 
ellipsoid of minimal volume containing $B$. We define a two-parameter family of 
Euclidean norms as follows: Given $x\in S$, let $x^*$ be the (or a fixed) support 
functional of $B$ at $x$ and set $H_x=\ker x^*$, so that each $z\in X$ can be 
written as $z=tx+u$, with $t\in\mathbb R$ and $u\in H_x$ --- of course 
$t=\langle x^*, z\rangle$ and $u=z-tx$. Then, for each $b>0$ we set
\begin{equation}
	|z|_b^x= \sqrt{t^2+ b^2|u|^2}\qquad\text{and}\qquad E_b^x=\{z\in X: |z|_b^x\leq 1\}.
\end{equation}
Note that  $a<b\implies E_b^x\subset E_a^x$ for fixed $x\in S$ and that for 
each fixed $b>0$, the norms $|\cdot|_b^x$ are all uniformly equivalent and 
uniformly equivalent to $\|\cdot\|$. Hence it suffices to see that there 
exists $b>0$ such that $B\subset E_b^x$ for all $x\in S$. But if we assume 
$\delta_X(\eps)\geq (c\eps)^2$ for some $c>0$ it is clear from the Lemma 
that if $x\in S$ and  $u\in H_x$ are such that $z=tx+u$ belongs to $B$, then
$$
t^2+ c^2|u|^2\leq  t^2+ c^2\|u\|^2\leq t^2+ \delta(\|u\|)\leq  1.
$$
It follows that $|\cdot|^x_c\leq \|\cdot\|$ for all $x$ and since $|x|^x_c=1$ 
we have that the ellipsoid $E^x_c$ is outer at $x$, which is enough.

\begin{remark}
The proof actually shows that a Banach space isomorphic to a Hilbert space has 
MUC of power type 2 if and only if the points of its unit sphere admit uniformly 
bounded outer ellipsoids and that it has MUS of power type 2 if and only if the 
points of the unit sphere admit inner ellipsoids whose intersection contains a 
neighbourhood of the origin.
\end{remark}

\section{Uniform micro-semitransitivity}\label{sec:UMST}

\begin{definition}
A Banach space $X$ (or its norm) is said to be uniformly micro-semi\-tran\-sitive 
(UMST) if for every $\eps>0$ there exists $\delta>0$ such that whenever $x,y\in S$ 
satisfy $\|x-y\|<\delta$ there is $T\in \Aut_1(X)$ such that $y=Tx$ with $\|T-{\bf I}_X\|<\eps$.
\end{definition}

It is clear that UMST $\implies$ BST  $\implies$ ST, see \cite[Remark 2.4]{CDKKLM} or \cite[Lemma 2.3]{BRP}.

The interest in UMST stems from the facts that all 1-complemented FD subspaces 
of a separable and transitive Banach space are UMST (\cite[Theorem 3.2]{wheeling}) 
and that the only previously known UMST norms were the Euclidean norms, which 
made it conceivable that UMST would be a characteristic property of Hilbert spaces.   

Unfortunately this is not the case:
in this Section we prove that every $C^2$ norm on the plane whose dual is also $C^2$ turns out to be UMST.

The remainder of this section is devoted to proving the following.

\begin{theorem}\label{th:UMST}
Every $C^2$ norm on the plane whose unit sphere has strictly positive 
curvature at every point is uniformly micro-semitransitive.
\end{theorem}

The following ``uniform'' version of Lemma~\ref{lem:dA<dB} is the key 
ingredient of the proof of Theorem~\ref{th:UMST}. It will be applied 
in due course to compare the unit sphere of the given norm with its 
image under a {\em slightly-changing} automorphism.

\begin{lemma}\label{lem:UMST}
Let $\Phi,\Gamma:\R \To \R^2$ be twice differentiable curves parameterized 
by arc length, $K>0$ and $\delta\leq \sqrt{2}/K$. Assume that:
\begin{itemize}
\item[{\rm(a)}] $\Phi$ and $\Gamma$ are tangent at $s=0:\,$ $\Phi(0)=\Gamma(0)=p$ and $ \Phi'(0)=\Gamma'(0) $. 

\item[{\rm(b)}] If $q=\Phi(s), r=\Gamma(t)$ for $|s|,|t|\leq\delta$, 
then $0< \varkappa(\Gamma, r) <\varkappa(\Phi, q)\le K$.
\end{itemize}
Then $\Phi[-\delta,\delta]$ and $\Gamma[-\delta,\delta]$ meet only 
at $p:\,$ $\{\Phi(s): |s|\leq \delta\}\cap \{\Gamma(t): |t|\leq \delta\}=\{p\}$.

\end{lemma}

\begin{proof} We may assume that $p$ is the origin and $\Phi'(0)=\Gamma'(0)=e_1$. 
Since $\varkappa(\Phi, q)=\|\Phi''(s)\|_2$ for $q=\Phi(s)$. Applying the Mean Value
Theorem to $\Phi'$ (and then the Implicit Function Theorem) we see that if 
$\sup_{|s|\leq \delta}\|\Phi''(s)\|_2 \leq K$ then there exist $-\delta \leq \alpha<0<\beta\leq \delta $ 
and a function $f:[\alpha,\beta]\To \R$ such that the set $\Phi[-\delta ,\delta ]$ 
agrees with the graph of $f$. For the same reason, taking (b) into account, we 
have that   $\Gamma[-\delta ,\delta ]$ agrees with the graph of certain function 
$g:[\tilde\alpha,\tilde\beta]\To \R$. We have $f(0)=g(0)=f'(0)=g'(0)=0$.  
We may and do assume that $f''$ and $g''$ are strictly positive in their respective 
domains. Let us check that $\Phi(0,\delta ]\cap \Gamma(0,\delta ]=\varnothing$. 
Otherwise there is $x>0$ in the common domain of $f$ and $g$ such that $f(x)=g(x)$.
Put $z=\inf\{x>0: f(x)=g(x)\}$: note that $f(z)=g(z)$ and that $z>0$ since $f(x)>g(x)$ 
for $x>0$ sufficiently small. Clearly $g'(z)\geq f'(z)$ and this implies that there 
is $0<v<z$ such that $f''(v)=g''(v)$. Put $u=\inf\{v>0: f''(v)=g''(v)\}$. Then $u>0$ 
and $f''(u)=g''(u)$ and $f'(v)\geq g'(v)\geq 0$ for $0\leq v\leq u$. Hence letting 
$q=(u,f(u)), r=(u, g(u))$, and applying (\ref{eq:kvsf''}), we obtain
$$
\varkappa(\Phi,q)=\frac{|f''(u)|}{\big(1+f'(u)^2\big)^{3/2}}\leq 
\frac{|g''(u)|}{\big(1+g'(u)^2\big)^{3/2}}=\varkappa(\Gamma,r),
$$
a contradiction.
\end{proof}

\begin{wrapfigure}{r}{0.45\textwidth}
  \begin{center}
    \includegraphics[width=0.33\textwidth]{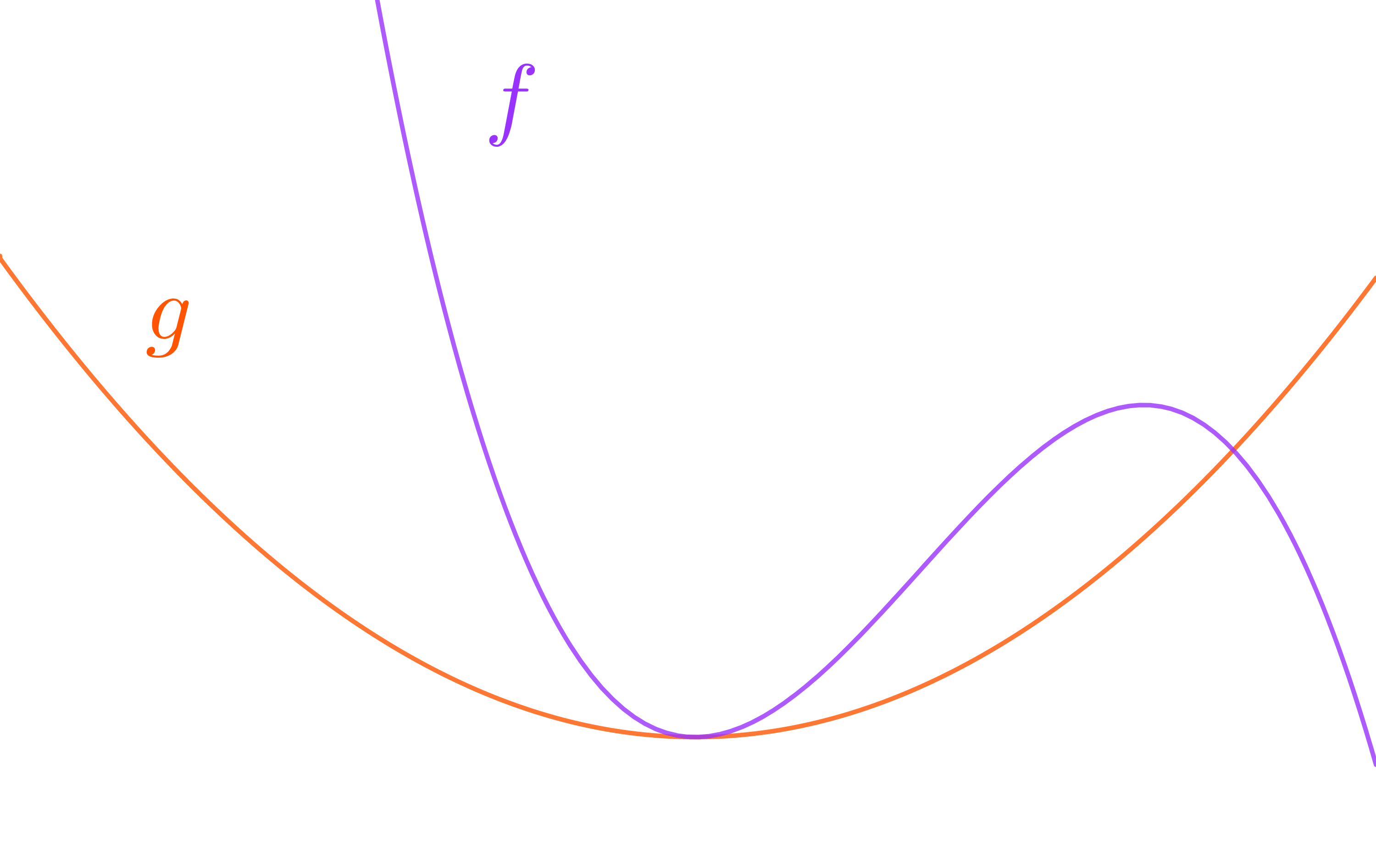}
  \end{center}
  \caption{Functions of the proof of Lemma~\ref{lem:UMST}}
\end{wrapfigure}

Let $\|\cdot\|$ be a (smooth) norm on the plane and let $S$ be its unit sphere 
with the standard (counterclockwise) orientation. If $\Sigma:\R\To\R^2$ is a 
regular parametrization of $S$ preserving the orientation and $a=\Sigma(t)$, put
$
a^\perp= {\Sigma'(t)}\big{/}{\|\Sigma'(t)\|},
$
that is, $a^\perp$ the only point in $S\cap H_a$ such that $a,a^\perp, -a$ is positively oriented.

It is almost obvious that if $a,b$ are close in $S$ and $T\in \Aut_1(X)$ is 
close to the identity and maps $a$ to $b$, then $Ta^\perp=\lambda b^\perp$, 
with $\lambda\in (0,1]$. From now on we denote by $L^{a b}_\eps$ the only 
linear endomorphism of $X$ that maps $a$ to $b$ and $a^\perp$ to $(1-\eps) b^\perp$, 
with the subscript omitted if $\eps=0$. This is actually an automorphism 
unless $\eps=1$ and, quite clearly,
 $$ 
L^{a b}_\eps= L^{b b}_\eps L^{a b}= L^{a b} L^{aa}_\eps .$$

\begin{lemma}\label{lem:abcd} Let $X$ be a 2-dimensional space with smooth norm.
\begin{itemize}
\item[{\rm (a)}]  $\|L^{aa}_\eps\|=1$ for all $a\in S$ and all $0\leq\eps\leq 1$.
\item[{\rm (b)}] If $X$ is strictly convex and  $\|L^{aa}_\eps(x)\|=\|x\|$ for 
some $0<\eps< 1$, then $x$ is proportional to $a$.
\item[{\rm (c)}]  $\|L^{aa}_\eps-{\bf I}_X\|\leq 2 \eps $ for all $a\in S$ and all $\eps>0$.
\item[{\rm (d)}] If the norm of $X$ is $C^2$, then there exists a constant $C$ 
(depending on $X$) such that $\|L^{ab}_\eps-{\bf I}_X\|\leq C\big(\|a-b\|+\eps\big)$ 
for all $a,b\in S$ and all $\eps>0$.
\end{itemize}
\end{lemma}

\begin{proof}
(a) and (b) are clear.
To prove (c) note that if $x=sa+ta^\perp$, then $s=\langle Ja, x\rangle$ since 
the coordinate $s(\cdot)$ is completely determined by $s(a)=1, s(a^\perp)=0$. 
In particular $|s|\leq \|x\|$. However $t(\cdot)$ does not agree with $J(a^\perp)$ 
unless $X$ is a Radon plane (one in which Birkhoff orthogonality is symmetric, do 
not worry if you have forgotten or never knew what this means; let us just add 
that while the spheres of Figure~8 are Radon, that of Figure~7 is not). 
Nevertheless, $|t|\leq 2\|x\|$. In particular,
$$
\|x- L^{aa}_\eps(x)\|=\|sa+ta^\perp-sa-t(1-\eps))a^\perp\|=\|t\eps a^\perp\|\leq 2\eps\|x\|.
$$

(d) Clearly, the map $a\in S\longmapsto a^\perp\in S$ is Lipschitz. For  $x=sa+ta^\perp$ we have
\begin{align*}
\|L^{ab}_\eps(x)-x\|= \|sb+(1-\eps)tb^\perp - sa+ta^\perp\|\leq |s|\|b-a\|+|t|\eps + |t|\|b^\perp- a^\perp\|,
\end{align*}
which is enough: $|s|\leq \|x\|$ and $|t|\leq 2\|x\|$.
\end{proof}

Note that if $X$ is assumed to be merely smooth, the proof yields 
$\|L^{ab}_\eps-{\bf I}_X\|\To 0$ as $\max\big(\|a-b\|,\eps\big)\To 0$.

We now analyze the variation of the curvature under certain linear transformations. 
The following result applies in particular when the curve is a sphere and $T=L^{a a}_\eps$.

\begin{lemma}\label{lem:1-e}
Let $a,v$ be linearly independent in $\R^2$ and $0<\eps<1$ and let $T$ be the only 
linear endomorphism such that $T(a)=a, T(v)=(1-\eps)v$. Assume $\Gamma$ is a regular 
curve such that $\Gamma(0)=a, \Gamma'(0)=v/\|v\|_2$ and let $\tilde{\Gamma}=T\,\Gamma$. 
Then $\tilde{\Gamma}(0)=a$ and 
$ 
\varkappa(\tilde{\Gamma},a)=(1-\eps)^{-2} \varkappa({\Gamma},a).$
\end{lemma}

\begin{proof}
After applying a suitable linear transformation (a rotation followed by scaling or 
viceversa) we may assume $a=e_2=(0,1)$ and $v=(1,m)$ in which case $T$ is implemented by the matrix
$$
\left(\begin{array}{cc}
1-\eps & 0 \\
-\eps m & 1
\end{array}\right)
$$
By the Implicit Function Theorem we can describe $\Gamma$ near $(0,1)$ as the 
graph of a certain funtion $f$ defined near $0$ and we then have
$$
\varkappa({\Gamma},a)= \frac{|f''(0)|}{\big(1+m^2\big)^{3/2}}
$$
see (\ref{eq:kvsf''}). Now, using the argument of $f$ as a parameter we have
$$
\tilde{\Gamma}(x)=\big( (1-\eps)x, -\eps m x+f(x)	\big), \quad 
\tilde{\Gamma}'(x) =\big( 1-\eps, -\eps m +f'(x)	\big), \quad 
\tilde{\Gamma}''(x)=\big( 0, f''(x)	\big).
$$
One the other hand (see \cite[Equation 8-34 on p. 410]{buck}), since $f'(0)=m$ we have
\begin{align}\label{eq:kparametrics}
\varkappa(\tilde{\Gamma},a) & = 	\frac{\sqrt{\|\tilde{\Gamma}'\|_2^2 \|\tilde{\Gamma}''\|_2^2 - \langle \tilde{\Gamma}', \tilde{\Gamma}''\rangle^2}}{ \|\tilde{\Gamma}'\|_2^3}\Bigg{|}_{x=0} \\ \nonumber
& = 
\frac{\sqrt{ (1-\eps)^2(1+m^2)f''(0)^2 - (1-\eps)^2 m^2 f''(0)^2 }}{ (1-\eps)^3 (1+m^2)^{3/2}}\\ 	
& = 
\frac{\sqrt{ (1-\eps)^2 f''(0)^2}}{ (1-\eps)^3 (1+m^2)^{3/2}}.\qedhere
\end{align}
\end{proof}

The core of the proof of Theorem~\ref{th:UMST} is the following piece:

\begin{lemma} Under the hypotheses of Theorem~\ref{th:UMST} for every $\eps>0$ 
there exists $\delta>0$ such that  $L^{ab}_\eps$ is contractive for all $a,b\in S$ 
such that $\|a-b\|<\delta$.
\end{lemma}

\begin{proof} If we assume the contrary then there exists $\eps>0$ such that for 
every $n\in\mathbb N$ there exist $a_n,b_n\in S$ with $\|a_n-b_n\|<1/n$ and 
$x_n\in S$ such that  $\|L^{a_n b_n}_\eps(x_n)\|> 1$.

Passing to subsequences without mercy we may assume that $(a_n), (b_n), (x_n)$ are 
convergent. Put $a=\lim_n a_n= \lim_n b_n$ and $x=\lim_n x_n$ and let us show that 
 $x=\pm a$. One has $L^{a_n b_n}_\eps(x_n)\To L^{a a}_\eps(x)$ and so 
$\|x\|= \|L^{a a}_\eps(x)\|= 1$ and Lemma~\ref{lem:abcd}(b) shows that $x=\pm a$. 
Now since $L^{ab}_{\eps}$ does not vary if we replace $a$ by $-a$ and $b$ by $-b$ 
we can even assume that $a=b=x$. Hence, it suffices to check the following:

\medskip

$(\natural)$\quad\emph{For each $\eps\in(0,\frac{1}{2})$ there exists $\delta>0$ 
(depending on $\eps$ and $X$) such that if $\|a-b\|<\delta$ then the only common 
point of $S$ and $L^{a b}_\eps[S]$ in the ball of radius $\delta$ and centre $b$ is $b$ itself}.
\medskip

With an eye in Lemma~\ref{lem:UMST} let $K$ be the supremum of the set
\begin{equation}\label{eq:elset}
\big\{ \varkappa\big( L^{a b}_\eps[S], x\big): a,b\in S,  \eps\in [0, \tfrac{1}{2}], x\in L^{a b}_\eps[S]\big\}.
\end{equation}
Since $L^{a b}_\eps[S]$ and $S$ are tangent at $b$ 
the following statement implies $(\natural)$ through Lemma~\ref{lem:UMST}:

\medskip

$(\sharp)$\quad\emph{For each $\eps\in(0,\frac{1}{2})$ there exists 
$\delta>0$ such that if $\|a-b\|<\delta$ then
$
\varkappa(L^{a b}_\eps[S], x)>\varkappa(S, y)
$, 
provided $x\in L^{a b}_\eps[S], y\in S$ and $\|x-b\|,\|y-b\|<\delta$}.
\medskip

We first observe that there is a constant $M$ such that 
$
{\|x\|}/{M}\leq \|L^{a b}_\eps(x)\|\leq M\|x\|
$
for all $x\in X$ and this implies that for each $\eps'>0$ there is $\delta_1=\delta_1(\eps')$ such that if $a,b\in S, \eps\in[0,\frac{1}{2}]$ and $x,y\in  L^{a b}_\eps[S]$, then 
$$
\|x-y\|\leq\delta_1\quad\implies\quad 
\Big|\varkappa(L^{a b}_\eps[S], x)- \varkappa(L^{a b}_\eps[S], y)\Big|<\eps'.
$$
In particular, if $x\in  L^{a b}_\eps[S]$ and $\|x-b\|<\delta_1$, then
$
\big|\varkappa(L^{a b}_\eps[S], x)- \varkappa(L^{a b}_\eps[S], b)\big|<\eps'.
$ 
In a similar vein, Lemma~\ref{lem:abcd}(d) implies that for each $\eps'>0$ there is $\delta_2(\eps')$ such that 
$$
\|a-b\|<\delta_2 \quad\implies\quad \Big|\varkappa(L^{a b}[S], b)- \varkappa(S, a)\Big|<\eps'.
$$
In particular, if $\|a-b\|\leq \min(\delta_1(\eps'),\delta_2(\eps'))$, then 
$$
 \Big|\varkappa(L^{a b}[S], b)- \varkappa(S, b)\Big|\leq
  \Big|\varkappa(L^{a b}[S], b)- \varkappa(S, a)\Big|+  \Big|\varkappa(S,a)- \varkappa(S, b)\Big|\leq 2\eps',
$$
hence,
\begin{equation}\label{eq:una}
\varkappa(L^{a b}[S], b)\geq  \varkappa(S, b) - 2\eps'.
\end{equation}

Now, we fix $\eps\in[0,\frac{1}{2}]$ and take $\eps'< \frac{1}{4}k\eps$, where  
$k>0$ is the infimum of the set in (\ref{eq:elset}).

Assume $\|a-b\|<\delta$, with $\delta\leq \min\big(\delta_1(\eps'), \delta_2(\eps') 	\big)$ and let us compare first $\varkappa(S, b)$ and $\varkappa(L^{a b}_\eps[S], b)$. According to Lemma~\ref{lem:1-e}
$$\varkappa(L^{a b}_\eps[S], b)= \frac{\varkappa(L^{a b}[S], b)}{(1-\eps)^2}.$$
Dividing  by $(1-\eps)^2$ in (\ref{eq:una}) and taking into account that $(1-\eps)^{-2}= 1+2\eps+ 3\eps^2+\cdots$ and that $2\eps\leq 1$ we get
\begin{gather*}
\frac{\varkappa(L^{a b}[S], b)}{(1-\eps)^2} \geq \frac{\varkappa(S, b) - 2\eps'}{(1-\eps)^2} 
\geq (1+2\eps) ( \varkappa(S, b) -2\eps')\\
\geq  \varkappa(S, b) - 2\eps' + 2\eps(k - 2\eps')
\geq  \varkappa(S, b) +k\eps.
\end{gather*}
Hence
$\varkappa(L^{a b}_\eps[S], b) \geq  \varkappa(S, b) +k\eps
$ provided $\|a-b\|<\delta$.
Therefore, if $x\in L^{a b}_\eps[S]$ has $\|x-b\|<\delta$ and  $y\in S$ has $\|y-b\|<\delta$ we obtain
$$
\varkappa(L^{a b}_\eps[S], x) \geq  \varkappa(S, y) +\tfrac{1}{2}k\eps.\qedhere
$$
\end{proof}

Let us present some  UMST norms on the plane in the light of Theorem~\ref{th:UMST}. The first set of examples uses polar coordinates.

\begin{example}
Let $g:\R\To (0,\infty)$ be a $C^2$ and $\pi$-periodic function satisfying the condition
\begin{equation}\label{eq:polar}
2g'(\theta)^2+ g(\theta)^2 - g(\theta)g''(\theta)>0
\end{equation}
for all $\theta$ (equivalently, for $0\leq\theta\leq\pi$). Then the formula  $\|v\|_g=r/g(\theta)$, 
where $(r,\theta)$ are the polar coordinates of $v$, defines a norm on the plane that satisfies the hypotheses of Theorem~\ref{th:UMST} and is therefore uniformly-micro-semitransitive. 

All norms on the plane satisfying the hypotheses of Theorem~\ref{th:UMST} arise in this way.
\end{example}

\begin{wrapfigure}{r}{0.45\textwidth}
  \begin{center}
    \includegraphics[width=0.32\textwidth]{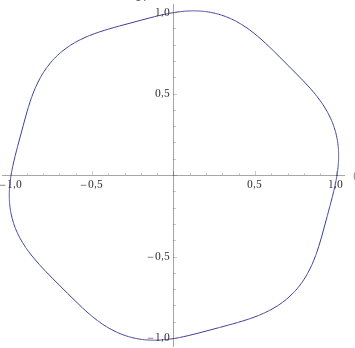}
  \end{center}
  \caption{The ``polar'' curve $r=1+\tfrac{1}{17}\sin(4\theta)$ is the sphere of a UMST norm. The resemblance to  Grandpa Pig is apparent.}
\end{wrapfigure}

Indeed if
 $\Gamma$ is a curve with polar representation $r=g(\theta)$, with $g:I\To (0,\infty)$ twice differentiable, then 
 the curvature of $\Gamma$ at the point of polar coordinates $(g(\theta),\theta)$ is given by
$$
\frac{|2g'(\theta)^2+ g(\theta)^2 - g(\theta)g''(\theta)|}{\big(g(\theta)^2+g'(\theta)^2\big)^{3/2}}
\hspace{220pt}$$
If in addition $g$ is $\pi$-periodic and satisfies (\ref{eq:polar}) then the bounded domain enclosed by $\Gamma$ is automatically convex ($\Gamma$ has no point of inflection) and thus $\Gamma$ is the unit sphere of a norm on $\R^2$ which is actually given by 
$\|v\|_g=r/g(\theta)$. Of course $\|\cdot\|_g$ is $C^2$ off the origin if and only if $g$ is $C^2$. The converse is obviously true.

\medskip

Note that if $g>g''$ (pointwise), then $(\ref{eq:polar})$ holds.
The first functions that satisfy this condition that come to mind are $g(\theta)= 1+ \eps \sin(n\theta)$ with $n$ even and $\eps<n^{-2}$; see Figure~7.

\begin{wrapfigure}{r}{0.45\textwidth}
  \begin{center}
    \vspace{-10pt}\includegraphics[width=0.30\textwidth]{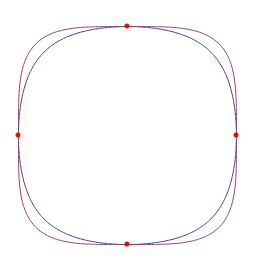}
  \end{center}
  \caption{The spheres of the norm $\|\cdot\|_4$ in red and  $\sqrt{\frac{1}{2}\|\cdot\|_4^2+\frac{1}{2}\|\cdot\|_2^2}$ in blue. The later is UMST, the former is not.}
\end{wrapfigure}

A different kind of example is the following: assume $\|\cdot\|$ is $C^2$ off the origin 
on $\R^n$. Given $\eps>0$ put $\|\cdot\|_{(\eps)}=\sqrt{\|\cdot\|^2+ \eps\|\cdot\|_2^2}$ This new norm is again $C^2$ and the ``new'' duality
map, which is given by
$
J_{\text{new}}= \tfrac{1}{2} d^2 \|\cdot\|_{(\eps)}^2= J_{\text{old}}+ \eps {\bf I}
$
is positive definite since $J_{\text{old}}$, being the Hessian of a convex function, is positive semi-definite. When $n=2$ this implies that 
 all points of the sphere have strictly positive curvature and so $\|\cdot\|_{(\eps)}$ is UMST.

We do not know if the converse of Theorem~\ref{th:UMST} is true, although
many of the implications of the proof are clearly reversible.
Assume $X$ is a UMST plane and let $\Sigma:I\To S$ be a parametrization of the 
unit sphere by the arc-length relative to the usual Euclidean norm.
It is relatively easy to see that there is $I_0\subset I$, whose complement 
has measure zero, where $\Sigma$ is twice differentiable and $\Sigma'':I_0\To X$ 
is uniformly continuous. In particular the curvature, which is defined on 
$\Sigma(I_0)$, has a continuous extension to $S$. Thus, the question seems to 
be whether UMST planes are {\em absolutely smooth} in the sense of \cite{banakh}.

\medskip

On the other hand, the theorem is crying out for a generalisation for 
arbitrary (finite) dimension. Of course we need a different hypothesis in ``the general case''. But
for a Banach space $X$ (necessarily isomorphic to a Hilbert space) the following statements are equivalent:
\begin{itemize}
\item  Both $X$ and $X^*$ are $C^2$-smooth off the origin.
\item  The function $\varphi:X\To \R$ given by $\varphi(x)=\frac{1}{2}\|x\|^2$ is $C^2$-smooth and $d^2_x\varphi$ is positive definite for all $x\in X$. 
\end{itemize}
(See \cite[Proof of Proposition~3.7]{wheeling}.) These in turn are equivalent to the hypothesis of Theorem~\ref{th:UMST} when $X$ has dimension 2 and we conjecture that they are equivalent to UMST in finite dimensions.

\end{document}